\documentclass[letterpaper, 10 pt, conference]{ieeeconf}

\IEEEoverridecommandlockouts                              
\usepackage{amssymb,amsmath,dsfont}
\usepackage{graphicx}
\usepackage{subfigure,xcolor}

\usepackage[english]{babel}
\usepackage{tikz,pgflibraryshapes,xcolor,xkeyval}
\usepackage{subfigure}
\usepackage{caption}

\newcommand {\col}{{\mathrm{col}}}

\newfont{\pseudocode}{cmtt10}

\newcommand{\pd}{{\partial}}

\newcommand{\Real}{\mathbb{R}}

\newcommand {\ess}{{\mathrm{ess}}}

\renewcommand{\Real}{\mathbb{R}}

\renewcommand{\Real}{\mathds{R}}

\newcommand{\rank}[1]{\mathrm{rank}\left(#1\right)}

\renewcommand{\col}[1]{\mathrm{col}\left(#1\right)}

\newcommand{\beq}{\begin{equation}}
\newcommand{\eeq}{\end{equation}}

\newcommand{\bbm}{\begin{bmatrix}}
\newcommand{\ebm}{\end{bmatrix}}

\newcommand{\bpm}{\begin{pmatrix}}
\newcommand{\epm}{\end{pmatrix}}

\newcommand{\bit}{\begin{itemize}}
\newcommand{\eit}{\end{itemize}}

\newcommand{\ben}{\begin{enumerate}}
\newcommand{\een}{\end{enumerate}}

\newcommand{\barr}{\begin{array}}
\newcommand{\earr}{\end{array}}

\newtheorem{thm}{Theorem}[section]
\newtheorem{lem}{Lemma}

\newtheorem{remark}{Remark}[section]
\newtheorem{assume}{Assumption}[section]
\newtheorem{prob}{Problem}[section]

\begin{document}

\title{\LARGE \bf Explicit Reduced-Order Integral Formulations of State and Parameter Estimation Problems for a Class of Nonlinear Systems}
\author{I.Yu. Tyukin, A.N. Gorban \thanks{I.Yu. Tyukin and A.N. Gorban  are with the University of Leicester, Department of Mathematics, University Road, LE1 7RH, Leicester, United Kingdom, e-mail:  {\tt\small I.Tyukin@le.ac.uk, ag153@le.ac.uk}}
}

\maketitle
\thispagestyle{empty}
\pagestyle{empty}

\begin{abstract} We propose a technique for reformulation of state and parameter estimation problems as that of matching explicitly computable definite integrals with known kernels to data. The technique applies for a class of systems of nonlinear ordinary differential equations and is aimed to exploit parallel computational streams in order to increase speed of calculations. The idea is based on the classical adaptive observers design. It has been shown that in case the data is periodic it may be possible to reduce dimensionality of the inference problem to that of the dimension of the vector of parameters entering the right-hand side of the model nonlinearly. Performance and practical implications of the method  are illustrated on a benchmark model governing dynamics of voltage in generated in barnacle giant muscle.
\end{abstract}

\section*{Notation}

Symbol $\|\cdot\|$ stands for the Euclidian norm. By
$\mathcal{K}$ we denote   the set of all strictly increasing
continuous functions $\kappa: \Real_{\geq 0}\rightarrow
\Real_{\geq 0}$ such that $\kappa(0)=0$. Consider a non-autonomous system $\dot{x}=f(x,p,t,u(t))$,  where $f:\Real^n\times\Real^d\times\Real\times\Real^l\rightarrow\Real^n$, $u:\Real\rightarrow\Real^l$
are continuous, $p\in\Real^d$ is the vector of parameters, and $f(\cdot,p,t,u)$ is locally Lipschitz; $x( \cdot \ ;t_0,x_0,p,[u])$ stands for
the unique maximal solution of the initial value problem: $x(t_0;t_0,x_0,p,[u])=x_0$. In cases when no confusion arises, we will refer to these solutions as $x(\cdot;t_0,x_0,[u])$,  $x(\cdot;x_0,[u])$,  or simply $x(\cdot)$. Solutions of the initial value problem above at $t$ are denoted as $x(t;t_0,x_0,p,[u])$,
$x(t;t_0,x_0,[u])$,  $x(t;x_0,[u])$,  or $x(t)$ respectively. Let $f:\Real\rightarrow\Real^n$, then $\|f(\tau)\|_{\infty,[t_0,t_0+T]}$ denotes the uniform norm of $f(\cdot)$ on $[t_0,t_0+T]$:  $\|f(\tau)\|_{\infty,[t_0,t_0+T]}=\ess \sup\{\|f(t)\|,t \in [t_0,t_0+T]\}$.

\section{Introduction}

Consider a system governed by nonlinear ordinary differential equations
\begin{equation}\label{eq:system0}
\dot{x}=f(x,p,t), \ x(t_0)=x_0,
\end{equation}
where $f:\Real^n\times\Real^k\times\Real\rightarrow\Real^n$ is \ continuous and locally Lipschitz wrt the variable $x$ function, and $p$ is the vector of unknown parameters. Let $[t_0,t_0+T]$ be an interval on which the solution $x(\cdot;t_0,x_0,p)$ of (\ref{eq:system0}) is defined. Let us further suppose that the system's state, $x(t;t_0,x_0,p)$, is not accessible for direct observation at any $t\in[t_0,t_0+T]$. One can, however, observe the values
\[
h(t,x(t;t_0,x_0,p)), \ h:\Real\times\Real^n\rightarrow\Real
\]
for every $t\in[t_0,t_0+T]$. Let the problem be to find  $p'\in \Real^k$, $x_0'\in\Real^n$ such that
\begin{equation}\label{eq:measurement:match}
\begin{array}{c}
h(t,x(t;t_0,x_0,p))=h(t,x(t;t_0,x_0',p'))\\
\mbox{for all}  \  t\in [t_0,t_0+T].
\end{array}
\end{equation}

This is a standard inverse problem, and many methods for finding solutions to this  problem have been developed to date (sensitivity functions \cite{SIAM:Ident}, splines \cite{Brewer:2008}, interval analysis \cite{Automatica:2008:Johnson},  adaptive observers \cite{Marino92},\cite{Besancon:2000}, \cite{Automatica:Farza:2009}, \cite{Automatica:Grip:2010},\cite{Tyukin:2011},\cite{Tyukin_2012:arx}  and particle filters and Bayesian inference methods \cite{Abarbanel:2009}). Despite these methods are based on different mathematical frameworks, they share a common feature: one is generally required to repeatedly  find numerical solutions of nonlinear ordinary differential equations (ODEs) over given intervals of time (solve the direct problem).

Notwithstanding the plausibility of numerical integration of systems of ODEs in algorithms for state and parameter estimation, this operation is an inherently sequential process. This constrains computational scalability of the problem, and as a result imposes limitations on the time required to derive a solution.  In order to overcome this limitation we propose  to cast the inverse problem above in an alternative, integral form. In particular, instead of finding numerical solutions of the initial value problem (\ref{eq:system0}) and matching the results to observed data, e.g. as (\ref{eq:measurement:match}), we search for a representation of the problem as
\begin{equation}\label{eq:integral_formulation_general}
\begin{split}
&y(t)-F(p,x_0,t)=0, \mbox{for all} \ t\in[t_0,t_0+T], \\
&F(p,x_0,t)=\int_{t_0}^t g(t,\tau,p,x_0)d\tau,
\end{split}
\end{equation}
where $g:\Real\times\Real\times\Real^k\times\Real^n\rightarrow \Real^d$ and $y:\Real\rightarrow \Real^d$ are functions that are explicitly computable from measurement data. Furthermore, we additionally require that if $p',x_0'$ is a solution of (\ref{eq:integral_formulation_general}) then it is also a solution of (\ref{eq:measurement:match}) and vise versa. 

In the next sections we  specify a class of systems for which such representation is possible. This class of systems is not as general as (\ref{eq:system0}) but is relevant enough in modelling applications. In Section \ref{sec:problem_statement} we define this class of systems and present general technical assumptions. This is followed by presentation of main results in Section \ref{sec:main_results}. The results are based on the periodicity assumption we impose on the data and also on known facts from the theory of adaptive observes \cite{Marino92},\cite{Lorea_2002}. In Section \ref{sec:example} we illustrate the approach with an example for state and parameter estimation of action potential models for neural membranes.

\section{Problem Formulation}\label{sec:problem_statement}

Consider the following class of systems
\begin{equation}\label{eq:1}
\begin{split}
\dot{x}&=A(\theta) x  + \Psi(y,t)\theta + v(y,q,\lambda,t)\\
\dot{q}&= P(y,\lambda,t) q + w(y,\lambda,t)\\
y&=C^{T}x, \ x(t_0)=x_0, \ q(t_0)=q_0,
\end{split}
\end{equation}
where $(x,q)$,  $x\in\Real^n$, $q\in\Real^d$ is the state vector,  $\theta\in\Real^m$, $\lambda\in\Real^p$ are parameters, $A(\theta)$ is an $n\times n$ real matrix, possibly dependent on $\theta$, and $C\in\Real^n$, $C=\col{1,0,\dots,0}$.  We assume that the following hold for (\ref{eq:1}):
\begin{assume}[General assumptions on (\ref{eq:1})] $\ \ \ \ \ \ \ \ \ \ $
\begin{itemize}
\item[A1)] the solution of (\ref{eq:1}) is defined on the interval $[t_0,t_0+T]$ (for some $T>0$, possibly dependent on $t_0$);
\item[A2)] the pair $A(\theta),C^T$, is  observable, that is
\[
\rank{\begin{array}{c}
                    C^{T}\\
                    C^T A(\theta)\\
                     \vdots\\
                     C^{T} A^{n-1}(\theta)
            \end{array}}=n;
\]
\item[A3)] $P(y,\lambda,t)$ and $\Psi(y,t)$ are  $d\times d$ and $n\times m$ real matrices of which the entries are continuous and differentiable functions; $P(y,\lambda,t)$ is diagonal:
    \[
    P(y,\lambda,t)=\mathrm{diag}\left(\alpha_1(y,\lambda,t),\dots,\alpha_d(y,\lambda,t) \right);
    \]
\item[A4)] $v:\Real\times\Real^d\times\Real^p\times\Real\rightarrow \Real^n$, $w:\Real\times\Real^p\times\Real\rightarrow\Real^d$ are continuous and differentiable functions.
\item[A5)] Exact values of parameters  $\theta$, $\lambda$ are unknown.
\end{itemize}
\end{assume}

Since the pair $A(\theta),C^T$ is observable there always is a coordinate transform $x\mapsto T(\theta,t,x)$, $q\mapsto q$ \cite{Marino92} rendering (\ref{eq:1}) into the following form
\begin{equation}\label{eq:system}
\begin{split}
\dot{x}&=A_0 x  + b \varphi(y,t)^{T} \tilde\theta(\theta) + \tilde{v}(y,q,\bar\lambda(\lambda,\theta),t)\\
\dot{q}&= \tilde P(y,\bar \lambda,t) q + \tilde w(y,\bar\lambda,t)\\
y&=C^{T}x, \ x(t_0)=x_0, \ q(t_0)=q_0,
\end{split}
\end{equation}
$b=(1,b_1,\dots,b_{n-1})^{T}$ is such that the polynomial $s^{n-1} + b_1s^{n-2}+\cdots+b_{n-1}$ is
Hurwitz, and $A_0=\left(\begin{array}{cc}0 & I_{n-1}\\
                           0 & 0\end{array}\right)$.
Functions
$\varphi:\Real\times\Real\rightarrow\Real^r$,
$\tilde{v}:\Real\times\Real^k\times\Real\rightarrow\Real^n$, $\tilde w: \Real\times\Real^k\times\Real\rightarrow \Real^d$, $\tilde P:\Real\times\Real^k\times\Real\rightarrow \Real^{d\times d}$ are
continuous and differentiable, $\tilde P$ is diagonal,  and $\tilde\theta\in\Real^r$,
$\bar\lambda\in\Real^k$ are parameters.

Furthermore, noticing that the variable $y$ is defined and known for all $t\in[t_0,t_0+T]$, and  $y(\cdot)$ is continuous one can express the solution $q(t;q_0,\bar\lambda,[y])$ on $[t_0,t_0+T]$ in the closed form as follows:
\[
\begin{split}
q(t;q_0,\bar\lambda,[y])&=e^{\int_{t_0}^t \tilde P(y(\tau),\bar\lambda,\tau)d\tau}q_0 +\\
&\int_{t_0}^t e^{\int_{\tau}^t \tilde P(y(s),\bar\lambda,s)ds} \tilde w(y(\tau),\bar\lambda,\tau) d\tau
\end{split}
\]
Denoting  $\tilde{\lambda}=\mathrm{col}(\bar\lambda,q_0)$, $g(y,\tilde{\lambda},t)=\tilde{v}(y,q(t;q_0,\bar\lambda,[y]),\bar\lambda,t)$ we therefore arrive at the transformed equations of (\ref{eq:system}):
\begin{equation}\label{eq:system1}
\begin{split}
\dot{x}&=A_0 x  + b \varphi(y,t)^{T} \tilde\theta(\theta) + g(y,\tilde{\lambda},t)\\
y&=C^{T}x, \ x(t_0)=x_0.
\end{split}
\end{equation}

The core problem we are interested in (\ref{eq:system1}) is as follows:
\begin{prob}\label{prob:inferrence} Let (\ref{eq:system1}) be given, and its solutions are defined on $[t_0,t_0+T]$. Suppose that all functions in the right-hand side of (\ref{eq:system}) are known, but  true values of $x_0$, $\tilde\theta$, $\tilde\lambda$ are unknown. Infer the values of $x(t_0), \tilde\theta(\theta), \tilde{\lambda}$ from the measurements of $y(t;t_0,x_0,\tilde{\theta},\tilde{\lambda})=C^{T}x(t;t_0,x_0,\tilde{\theta},\tilde\lambda)$ over $[t_0,t_0+T]$.
\end{prob}

 The question is if there is an equivalent integral formulation such as e.g. (\ref{eq:integral_formulation_general}) of this problem for (\ref{eq:system1})? If such an integral formulation exists then whether a reduced-complexity version of this formulation can be stated so that the dimension of the parameter vector in the reduced formulation is smaller that that of in the original problem? Answers to these questions are provided in the next section.

\section{Main Result}\label{sec:main_results}

\subsection{Indistinguishable parameterizations of (\ref{eq:system1})}

We begin with the following property of linear systems regarding input detectability (cf \cite{Tyukin_2012:arx})
\begin{lem}\label{lem:observer_inferrence}
Consider
\begin{equation}\label{eq:system_io}
\begin{split}
& \begin{array}{ll}
\dot{x}&=A x + u(t)+d(t),\\
y&=C^{T}x, \ x(t_0)=x_0, \ x_0\in\Real^n,
\end{array}
\end{split}
\end{equation}
where
\[
A=\left(\begin{array}{cc} \begin{array}{c}
                                a_1\\
                                \vdots\\
                                a_n
                          \end{array} & \begin{array}{c}
                                          I_{n-1}\\
                                           \\
                                            0
                                            \end{array}
         \end{array}\right), \ C=(1,0,\dots,0)^{T},
\]
and
$x, u,d:\Real\rightarrow\Real^n$, $u\in\mathcal{C}^1$,
$d\in\mathcal{C}$. Let $u(\cdot),\dot{u}(\cdot),d(\cdot)$ be bounded:
$\max\{\|u(t)\|,\|\dot{u}(t)\|\}\leq B, \
\|{d}(t)\|\leq \Delta_{\xi}$ for all $t\geq t_0$. Then the following hold:

\begin{itemize}
\item[1) ] if the solution of (\ref{eq:system_io}) is globally bounded for all $t\geq t_0$ then, for $T$ sufficiently large, there are $\kappa_1,\kappa_2\in\mathcal{K}$:
\[
\begin{split}
&\|y(\tau)\|_{\infty,[t_0,t_0+T]}\leq \varepsilon \Rightarrow  \ \exists \ t'(\varepsilon,x_0)\geq t_0: \\
&\left\|z_1(\tau) + u_1(\tau) \right\|_{\infty,[t',t_0+T]}\leq \kappa_1(\varepsilon) + \kappa_2(\Delta_{\xi}),
\end{split}
\]
where $z_1=(1,0,\dots,0) z$,
\begin{equation}\label{eq:system_io_filter}
\begin{split}
\dot z&= \Lambda  z +   G  u,\ \Lambda=\left(\begin{array}{ccc}
- b \  & \begin{array}{c}\vdots\\ \vdots \end{array} &
\begin{array}{c} I_{n-2}\\ 0
\end{array}\end{array}\right),\\
 G&=\left(\begin{array}{cc} - { b} & I_{n-1} \end{array}\right), \  z(t_0)=0,
\end{split}
\end{equation}
and  $ b=(b_1,\dots,b_{n-1})^{T}$:  real parts of the roots
of $s^{n-1}  + b_1 s^{n-2}+\cdots + b_{n-1}$ are negative.

\item[2)] if $ d(t)\equiv 0$, then $y(t)= 0$ for all $t\in[t_0,t_0+T]$ implies existence of $p\in\Real^{n-1}$
\begin{equation}\label{eq:identifiability}
(1,0,\dots,0)e^{\Lambda(t-t_0)} p+z_1(t)+u_1(t)=0
\end{equation}
for all  $t\in[t_0,t_0+T]$.
\end{itemize}
\end{lem}
Proof of Lemma \ref{lem:observer_inferrence} is provided in the Appendix.

According to Lemma \ref{lem:observer_inferrence} the following two sets of parameters, associated with every $\tilde\theta, \tilde\lambda$, need special consideration. The first set is defined as
\[
\begin{split}
&\mathcal{E}_0(\tilde{\theta},\tilde{\lambda},T)=\{(\theta',\lambda'), \ \theta'\in\Real^r,\lambda'\in\Real^{d+k} \ | \\
& b \varphi(y(t),t)^{T} (\theta'-\tilde\theta) + g(y(t),{\lambda}',t)-g(y(t),\tilde{\lambda},t)=0 \\
& \ \ \ \ \ \mbox{for all} \ t\in[t_0,t_0+T] \}.
\end{split}
\]
The set $\mathcal{E}_0(\tilde{\theta},\tilde{\lambda},T)$ contains all parameterizations of (\ref{eq:system1}) which are indistinguishable from each other providing that the values of  $x(t)$ are known for all $t\in[t_0,t_0+T]$. That is, if $x(t;t_0,x_0,\tilde\theta,\tilde\lambda)=x(t;t_0,x_0,\theta',\lambda')$ for all $t\in[t_0,t_0+T]$ then $(\theta',\lambda')\in\mathcal{E}_0(\tilde\theta,\tilde\lambda,T)$. Denote
\[
\begin{split}
&\eta(\tilde\theta,\tilde\lambda,\theta',\lambda',p,t)=\varphi(y(t),t)^{T} (\theta'-\tilde\theta) +  g_1(y(t),{\lambda}',t)\\
&-g_1(y(t),\tilde{\lambda},t)+\tilde{C}^T e^{\Lambda(t-t_0)}p + z_1(t;t_0,\lambda')-z_1(t;t_0,\tilde\lambda),
\end{split}
\]
where $\Lambda,\tilde C, z(t;t_0,\lambda')$ are defined as in (\ref{eq:system_io_filter}) with $u(t)$ replaced by $g(y(t),\lambda',t)$. The second set is defined as
\[
\begin{split}
&\mathcal{E}(\tilde\theta,\tilde\lambda,T)=\{(\theta',\lambda'), \ \theta'\in\Real^r,\lambda'\in\Real^{d+k} \ | \\ \exists \ &p(\tilde\theta,\tilde\lambda,\theta',\lambda')\in \Real^{n-1}: \\
 &\eta(\tilde\theta,\tilde\lambda,\theta',\lambda',p,t)=0 \ \mbox{for all} \ t\in[t_0,t_0+T] \}.
\end{split}
\]
In accordance with Lemma \ref{lem:observer_inferrence} the set $\mathcal{E}(\tilde\theta,\tilde\lambda,T)$ contains all parametrization of (\ref{eq:system1}) that are indistinguishable on the interval $[t_0,t_0+T]$ on the basis of accessing only the values of $y(x(t;t_0,x_0,\theta,\lambda))$. In other words, if $y(x(t;t_0,x_0,\tilde\theta,\tilde\lambda))=y(x(t;t_0,x_0',\theta',\lambda'))$ for all $t\in[t_0,t_0+T]$ then $(\theta',\lambda')\in\mathcal{E}(\tilde\theta,\tilde\lambda,T)$. If the set $\mathcal{E}(\tilde\theta,\tilde\lambda,T)$ contains more than one element then (\ref{eq:system1}) is not uniquely identifiable  on $[t_0,t_0+T]$ \cite{Distefano:1980}. Here, for simplicity, we will focus on systems (\ref{eq:system1}) that are uniquely identifiable on $[t_0,t_0+T]$:
\begin{assume}\label{assume:unique_ident} Sets $\mathcal{E}_0(\tilde\theta,\tilde\lambda,T)$ and $\mathcal{E}(\tilde\theta,\tilde\lambda,T)$ coincide and contain no more than one element.
\end{assume}

\subsection{Integral reduced-order formulation of the inverse problem for (\ref{eq:system1})}

Before we proceed with presenting an equivalent integral formulation of Problem \ref{prob:inferrence} let us first introduce several additional components and corresponding technical assumptions. Let $l\in\Real^n$ be a vector satisfying the following condition:
\[
P(A_0+l C^T)+(A_0+l C^T)^T P = -Q, \ Pb=C,
\]
where $P,Q$ are some symmetric positive definite matrices. According to the Meyer-Kalman-Yakubovich-Popov lemma, such vector will always exist since the polynomial $s^{n-1}+b_1 s^{n-2}+ \cdots + b_{n-1}$ is Hurwitz.

Consider
\begin{equation}\label{eq:error_dynamics}
\frac{d}{dt}\left(\begin{array}{c} \xi_1\\ \xi_2\end{array}\right)=\left(\begin{array}{cc} A_0+ l C^T & b\varphi(y(t),t)\\
                                    - \varphi(y(t),t) C^T & 0 \end{array}\right) \left(\begin{array}{c} \xi_1\\ \xi_2\end{array}\right),
\end{equation}
and let $\Phi(t,t_0)$ be its corresponding normalized fundamental solutions matrix: $\Phi(t_0,t_0)=I_{n+r}$.

\begin{thm}\label{theorem:integral_formulation} Consider (\ref{eq:system1}) and suppose that Assumption \ref{assume:unique_ident} holds. Let $y(\cdot)$, $\varphi(y(\cdot),\cdot)$, $g(y(\cdot),\lambda,\cdot)$ be $T$-periodic on $[t_0,\infty]$ for all $\lambda$, and the function $\varphi(y(\cdot),\cdot)$ satisfy:
\[
\int_{t_0}^{t_0+T}\varphi(y(\tau),\tau) \varphi(y(\tau),\tau)^{T}d\tau \geq \delta I_r, \ \delta >0.
\]

Then the following statements are equivalent
\begin{itemize}
\item [1)]  $\hat{y}(\lambda',t)=y(t)$ for all $t\in [t_0,t_0+T]$, where $\hat{y}:\Real^{d+k}\times\Real\rightarrow\Real$:
\begin{equation}\label{eq:integral_formulation_canonic}
\begin{split}
&\hat{y}(\lambda',t)=(1 \ 0 \ \dots \ 0)\big(\Phi(t,t_0)R(\lambda')+\Phi(t,t_0)\times\\
& \int_{t_0}^{t}\Phi(\tau,t_0)^{-1} \left(\begin{array}{c} g(y(\tau),\lambda',\tau)-ly(\tau)\\ y(\tau)\varphi(y(\tau),\tau) \end{array}\right)d\tau \big)\\
&R(\lambda')=(I_{n+r}-\Phi(t_0+T,t_0))^{-1}\Phi(t_0+T,t_0)\times\\
& \int_{t_0}^{t_0+T}\Phi(\tau,t_0)^{-1} \left(\begin{array}{c} g(y(\tau),\lambda',\tau)-l y(\tau)\\ y(\tau)\varphi(y(\tau),\tau) \end{array}\right)d\tau.
\end{split}
\end{equation}
\item [2)] $(1 \ 0  \ \cdots \ 0)x(t;t_0,x_0,\tilde\theta,\lambda')=y(t)$ for all $t\in [t_0,t_0+T]$.
\end{itemize}

Furthermore, the values of $x_0$, $\tilde\theta$ satisfy
\begin{equation}\label{eq:initial_conditions}
\left(\begin{array}{c}x_0\\
                    \tilde\theta
      \end{array}
\right)=R(\lambda').
\end{equation}

\end{thm}

\begin{proof} Let us first show that 1) $\Rightarrow$ 2). Recall (see e.g. \cite{Lorea_2002}) that assumptions of the theorem imply  existence of positive numbers $\rho,D>0$:
\[
\|\Phi(t,t_0')\|\leq D e^{-\rho(t-t_0')} \ \mbox{for all} \ t\geq t_0',  \ t,t_0'\in[t_0,\infty).
\]
Hence there are no zero eigenvalues of the matrix  $I_{n+r}-\Phi(t_0+T,t_0)$, and $(I_{n+r}-\Phi(t_0+T,t_0))^{-1}$ exists.

Consider $\chi=(\chi_1,\chi_2)$:
\begin{equation}\label{eq:estimates}
\begin{split}
& \frac{d}{dt}\left(\begin{array}{c} \chi_1\\ \chi_2\end{array}\right)= \left(\begin{array}{cc} A_0+ l C^T & b\varphi(y(t),t)\\
                                    - \varphi(y(t),t) C^T & 0 \end{array}\right) \left(\begin{array}{c} \chi_1  \\ \chi_2\end{array}\right)  \\
&+ \left(\begin{array}{c} g(y(t),\lambda',t)-ly(t)\\ y(t)\varphi(y(t),t) \end{array}\right)
\end{split}
\end{equation}
It is clear that solutions of (\ref{eq:estimates}) are defined for all $t\geq t_0$ providing that the definition of $y(\cdot)$, $g(y(\cdot),\lambda',\cdot)$, and $\varphi(y(\cdot),\cdot)$ are extended (periodically) on the interval $[t_0,\infty)$. Introduce the function $\zeta(\cdot)=(x(\cdot,t_0,x_0,\tilde\theta,\tilde\lambda),\tilde\theta)$ (in which the domain of the function $x(\cdot,t_0,x_0,\tilde\theta,\tilde\lambda)$ definition is extended to $[t_0,\infty)$), and consider the difference
\[
\xi=\chi-\zeta.
\]
Dynamics of $\xi$ satisfy (\ref{eq:error_dynamics}) with $\xi_1(t_0)=\chi_1(t_0)-x(t_0)$, $\xi_2(t_0)=\chi_2(t_0)-\tilde\theta$. Moreover, $\hat{y}(\lambda',t)=C^{T}\chi_1(t)$ for all $t \in[t_0,t_0+T]$ (or in $[t_0,\infty)$ if $\hat{y}(\lambda',\cdot)$ is periodically extended on $[t_0,\infty)$).

Let $\hat{y}(\lambda',t)\equiv y(t)$. This implies that $\chi_2-\tilde\theta = \mbox{const}$ for all $t\in[t_0,t_0+T]$. Hence according to Lemma \ref{lem:observer_inferrence} $(\chi_2(t_0),\lambda')$ belong to $\mathcal{E}(\tilde\theta,\tilde\lambda,T)$. Given that sets $\mathcal{E}(\tilde\theta,\tilde\lambda,T)$ and $\mathcal{E}_0(\tilde\theta,\tilde\lambda,T)$ coincide and contain just one element, $\tilde\theta,\tilde\lambda$, we conclude that $\chi_2(t_0)=\tilde\theta$, $\lambda'=\tilde\lambda$. 

Notice that  $\lim_{t\rightarrow \infty}\xi(t)=0$ for all $\chi(t_0)$, and that
\begin{equation}\label{eq:stable_periodic}
\begin{split}
&\Phi(t,t_0)R(\lambda')+\Phi(t,t_0)\times\\
& \int_{t_0}^{t}\Phi(\tau,t_0)^{-1} \left(\begin{array}{c} g(y(\tau),\lambda',\tau)-ly(\tau)\\ y(\tau)\varphi(y(\tau),\tau) \end{array}\right)d\tau \big)
\end{split}
\end{equation}
is the unique exponentially stable periodic solution of (\ref{eq:estimates}). This implies that (\ref{eq:initial_conditions}) holds.

Let us show that 2) $\Rightarrow$ 1). Let $\tilde\theta,\lambda'$ be parameters for which the following identity folds $y(x(t;t_0,x_0,\tilde\theta,\lambda'))=y(t)$ for all $t\in[t_0,t_0+T]$. Consider the function $\zeta(\cdot)$ defined earlier.  Given that (\ref{eq:stable_periodic}) is the unique exponentially stable periodic solution of (\ref{eq:estimates}),  that $\lim_{t\rightarrow\infty}\zeta(t)=0$ for arbitrary choice of initial conditions (i.e. vectors $\tilde\theta$, $x(t_0)$, and $\chi_1(t_0)$, $\chi_2(t_0)$) and that $\zeta(t)\equiv 0$ if $\chi_1(t_0)=x_0$, $\chi_2(t_0)=\tilde\theta$,  one concludes that $\hat{y}(\lambda',t)=y(x(t;t_0,x_0,\tilde\theta,\lambda'))=y(t)$ for all $t\in[t_0,t_0+T]$.
\end{proof}

\begin{remark} One may argue that it is, in principle, possible to obtain integral formulations of the corresponding inverse problem without using adaptive observer-inspired structures. Note, however, that since the original matrix $A(\theta)$ is allowed to depend on unknown parameters $\theta$, explicit expressions of solutions of (\ref{eq:system}) will involve extra nonlinearly parameterized terms, $e^{A(\theta)(t-t_0)}$. If closed-form expressions are applied to (\ref{eq:system1}) then the drawback is that the overall unknown parameters vector is $(x_0,\tilde\theta,\tilde\lambda)$, and its dimension is $n+r+d+k$. In the proposed solution  dimension of the unknown parameters vector is reduced to $d+k$ which is advantageous for systems with large number of unknowns.
\end{remark}

\begin{remark} The uncertainty reduction achieved in the proposed method is due to the assumption that all functions in the right-hand side of (\ref{eq:system1}) are $T$-periodic. Whereas such periodicity assumptions may not always hold, they are not particularly difficult to satisfy (at least approximately) in the laboratory conditions.
\end{remark}

\begin{remark}\label{rem:discrete} Instead of dealing with continuous-time signals, $y(t)$, one may re-formulate the above results for data sampled at an $N$ discrete points $\{t_i\}$ in $[t_0,t_0+T]$. In this case sets $\mathcal{E}_0$, $\mathcal{E}$ will need to be re-defined so that the corresponding identities hold at a finite number of points $\{t_i\}$ rather than for all $t\in[t_0,t_0+T]$. Discrete extension of the theorem allows straightforward formulation of the inference problem as
\begin{equation}\label{eq:discrete}
\tilde\lambda=\arg \min_{\lambda\in\Real^r} \sum_{i=1}^N (\hat{y}(\lambda,t_i)-y(t_i))^2
 \end{equation}
which bears some similarity with \cite{Kuhl_2011,Pavlov_2013}. Here, however, no discretization of the original continuous-time dynamical model is required and  ${\pd \hat y(\lambda,t_i)}/{\pd \lambda}$ are computable as definite integrals.
\end{remark}

\section{Example}\label{sec:example}

Consider the following system:
\begin{equation}\label{eq:Morris-Lecar}
\begin{split}
\dot{x}=&-g_{Ca}m_{\infty}(x)(x-E_{Ca})-g_{K}q(x-E_K)\\
&-g_L(x-E_L)+I\\
\dot{q}=&-\frac{1}{\tau(x)}q+\frac{w_\infty(x)}{\tau(x)},\\
y=&x,
\end{split}
\end{equation}
where
\[
\begin{split}
m_\infty(x)&=0.5\left(1+\tanh\left(\frac{x-V_1}{V_2}\right)\right)\\
w_\infty(x)&=0.5\left(1+\tanh\left(\frac{x-V_3}{V_4}\right)\right)\\
\tau(x)&=T_0 \left(\cosh\left(\frac{x-V_3}{2V_4}\right)\right)^{-1}
\end{split}.
\]
Equations (\ref{eq:Morris-Lecar}) model dynamics of voltage oscillations generated in barnacle giant muscle fiber \cite{Morris_Lecar}. Variable $x$ is the measured voltage, $q$ is the recovery variable. The values of $E_{Ca}$, $E_K$, $E_L$ are normally known ($E_{Ca}=-100$, $E_K=70$, $E_L=50$); other parameters may vary from one cell to another.

It is clear that equations (\ref{eq:Morris-Lecar}) are of the form (\ref{eq:system}). Moreover, if the model operates in the oscillatory regime then the right-hand side is periodic in $t$, including the variable $q$. In addition the integral
\[
\int_{t_0}^{t_0+T}-\frac{1}{\tau(x(s))}ds < 0,
\]
where if $T$ is the period of oscillations, for practically relevant values of $T_0,V_3,V_4$. Assuming that observations are taking place when the system's solution are on (or sufficiently near) the stable period orbit we can express the variable $q(t)$  as follows:
\[
\begin{split}
q(t)=&e^{\int_{t_0}^{t}-\frac{1}{\tau(x(s))}ds} q_0 + \int_{t_0}^{t} e^{\int_{z}^{t}-\frac{1}{\tau(x(s))}ds}\frac{w_\infty(x(z))}{\tau(x(z))}dz\\
q_0=& (1-e^{\int_{t_0}^{t_0+T}-\frac{1}{\tau(x(s))}ds})^{-1}\times\\
&\int_{t_0}^{t_0+T} e^{\int_{z}^{t}-\frac{1}{\tau(x(s))}ds}\frac{w_\infty(x(z))}{\tau(x(z))}dz.
\end{split}
\]
This brings equations (\ref{eq:Morris-Lecar}) into the form (\ref{eq:system1}) with parameters $\tilde\theta=(g_L,I)$, and $\tilde\lambda=(V_1,V_2,V_3,V_4,T_0,g_{Ca},g_K)$.

For the purpose of illustration we set the values of parameters $\tilde\theta$, $\tilde\lambda$ as specified in Table \ref{table:parameters}.
\begin{table}
\caption{True (first row) and Estimated (second) parameter values of (\ref{eq:Morris-Lecar})}\label{table:parameters}
\label{table_example}
\begin{center}
Vector $\tilde\lambda=(V_1,V_2,V_3,V_4,T_0,g_{Ca},g_K)$\\
\vspace{2mm}
\begin{tabular}{|c|c|c|c|c|c|c|}
\hline
 $V_1$ & $V_2$ & $V_3$ & $V_4$ & $T_0$ & $g_{Ca}$ & $g_K$\\
\hline
 $1$ & $15$ & $-10$ & $14.5$ & $3$ & $-1.1$ & $2$\\
\hline
 $0.95$ & $15.08$ & $-10.15$ & $14.44$ & $3.04$ & $-1.12$ & $2.02$\\
\hline
\end{tabular}\\
\vspace{5mm}
Vector $\tilde\theta=(g_L,I)$\\
\vspace{2mm}
\begin{tabular}{|c|c|}
\hline
$g_L$  & $I$ \\
\hline
 $-0.5$  & $10$  \\
\hline
 $-0.539$  & $10.65$ \\
\hline
\end{tabular}
\end{center}
\end{table}
For the data generated at these parameter values the system is uniquely identifiable, and hence Assumption \ref{assume:unique_ident} holds. According to Theorem \ref{theorem:integral_formulation}, the problem of finding the values of $\tilde\theta,\tilde\lambda$ can be now formulated as that of matching the function $\hat{y}(\lambda',t)$ defined in (\ref{eq:integral_formulation_canonic}) to $y(t)$ over $[t_0,t_0+T]$. And in view of Remark \ref{rem:discrete} it reduces to
solving the unconstrained program (\ref{eq:discrete}).

In order to evaluate $\hat{y}(\lambda',t)$, as a function of parameter $\lambda'$ at a given $t$ one needs to know the fundamental solutions matrix $\Phi(t,t_0)$ for all $t\in[t_0,t_0+T]$. In this example this matrix was constructed numerically (using Dormand-Prince method and with fixed step size $0.0002$) from linearly independent solutions of
\begin{equation}\label{eq:example_fundamental}
\dot{z}=\left(\begin{array}{ccc}-l  & y(t) & 1\\
                                   -y(t) & 0 & 0\\
                                   -1 & 0 & 0  \end{array}\right)z, \ l=1
\end{equation}
starting from $(1,0,0)^{T}$, $(0,1,0)^{T}$, and $(0,0,1)^{T}$.

Points $t_i$ in (\ref{eq:discrete}) were evenly spaced with $t_{i+1}-t_i=0.04$, and the BFGS quasi-Newton method was used to find a numerical estimation of the solution of (\ref{eq:discrete}). For computational convenience, instead of looking for $V_1,V_2$ directly we were estimating ratios $1/V_2$ and $V_1/V_2$ respectively. Similarly, as follows from (\ref{eq:example_fundamental}), the estimate of parameter $I$ is not the value of $\tilde{\theta}_2$ but rather is the sum $\tilde{\theta}_2-\tilde{\theta}_1 50$.
We run the method for $12000$ iterations, and results  of the estimation are shown in Table \ref{table:parameters} and Fig. \ref{fig: estimates}.
\begin{figure}
\centering
\includegraphics[width=0.85\linewidth]{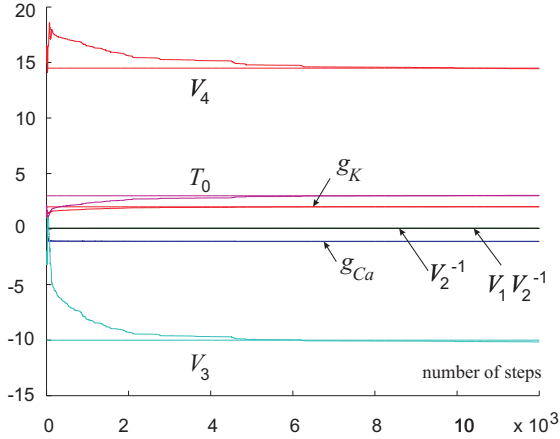}
\caption{Estimates and true values of $g_L$, $I$, $V_1/V_2$, $1/V_2$, $V_3$, $V_4$, $T_0$, $g_{Ca}$, $g_K$}\label{fig: estimates}
\end{figure}
In order to verify the quality of parameter estimation we run (\ref{eq:Morris-Lecar}) with both estimated and true values of parameters. Results of this simulation are show in Fig. \ref{fig:M-L-trajectoris}, upper panel. Note that frequency of the estimated $x(t)$ is higher than that of the measured data. This explains noticeable difference between trajectories at the end of the interval. In order to compensate for this difference we adjusted parameter $\tilde{\theta}_2$ (regulating the frequency of oscillations in the original model) by $-0.07$. Simulated trajectory of (\ref{eq:Morris-Lecar}) after this adjustment is shown in Fig. \ref{fig:M-L-trajectoris}, lower panel.
\begin{figure}
\centering
\includegraphics[width=0.95\linewidth]{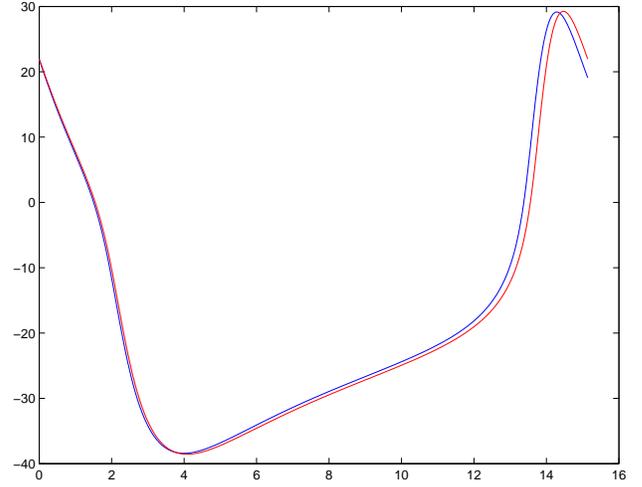}

\vspace{5mm}

\includegraphics[width=0.95\linewidth]{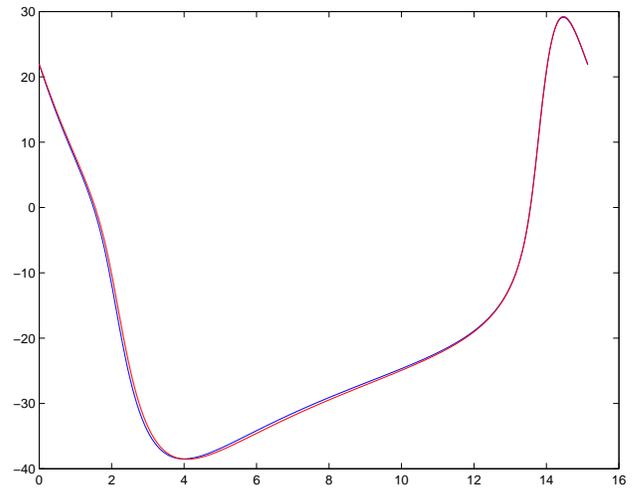}
\caption{Trajectories $x(t)$ of (\ref{eq:Morris-Lecar}) with true values of $\tilde\theta,\tilde\lambda$ (red curves) and estimated values of $\tilde\theta,\tilde\lambda$ from Table \ref{table:parameters} (blue curves). The upper panel shows the case when no adjustments to estimated parameters were made. The lower panel illustrates how the reconstructed $x(t)$ changes when the parameter $\tilde{\theta}_2$ regulating the frequency of oscillations is slightly adjusted by $-0.07$.}\label{fig:M-L-trajectoris}
\end{figure}
It is worth noticing that even though both estimated and simulated $x(t)$ are matching reasonably well there are still errors. The origin of these errors is likely to be 1) due to numerical errors in estimating the matrix $\Phi(t,t_0)$, and 2) due to the ill-conditioning of the original problem. Indeed, as Fig. \ref{fig:M-L-ill} suggests, there is a long shallow valley in a vicinity of the optimum.
\begin{figure}
\centering
\includegraphics[width=0.95\linewidth]{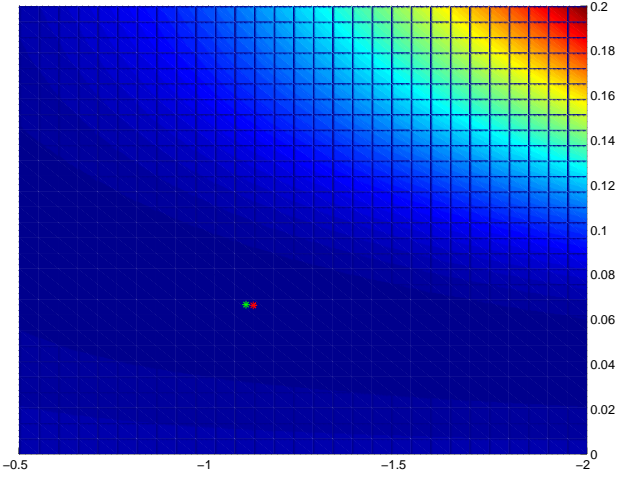}
\caption{Estimation error $\sum_{i=1}^N (\hat{y}(\tilde\lambda,t)-y(t))^2$ plotted as a function of $1/V_2$, $g_{Ca}$. Red star marks estimated $(g_{Ca},1/V_2)$, and green star corresponds to the true values of $(g_{Ca},1/V_2)$  }\label{fig:M-L-ill}
\end{figure}

The estimation took approximately $1$ hour on a standard PC in MATLAB. We observed that most of the time was spent in the calculations of $\frac{\pd \hat{y}}{\pd \tilde\lambda}$ which is not surprising given the integration (\ref{eq:integral_formulation_canonic}) was performed over a relatively dense and uniform grid of points. On the other hand, this indicates that in this and similar cases scalability of the procedure is expected to grow nearly linear with dimension of $\tilde\lambda$. This will be tested in experiments in future.

\section{Conclusion}

We presented a technique for explicit reduced-order integral reformulation of inverse problems for a class of nonlinear systems. The technique is aimed at using parallel computational streams and is based on the ideas of adaptive observers. It has been shown that the method allows to reduce dimensionality of the problem to that of the dimension of the vector of parameters entering the right-hand side of the model nonlinearly. In order to test the viability of the method a benchmark model governing dynamics of voltage in generated in barnacle giant muscle fiber has been chosen. The method performed well in this problem which, if coupled with inherent scalability of the procedure, enables to hope that the very same inference technology can be used  successfully for efficient fitting of other models to data too.

\bibliographystyle{plain}
\bibliography{theory_for_parallel}

\appendix
\section{Appendix}\label{app:proofs}

\begin{lem}\label{lem:system_io_inferrence_1order} Consider $\dot{y}=k y + u(t)+d(t)$, $k\in\Real$, $u,d:\Real_{\geq t_0}\rightarrow\Real$, $u\in\mathcal{C}^1$, $ d\in\mathcal{C}^{0}$, and let $\max\{|u(t)|,|\dot{u}(t)|\}\leq B$, $|d(t)|\leq \Delta_\xi$. Finally, let $T,\varepsilon$ be non-negative real numbers such that $T>\sqrt{\varepsilon}$. Then
\[
\begin{split}
&\|y\|_{\infty,[t_0,t_0+T]}\leq \varepsilon \Rightarrow  \\
& \|u\|_{\infty,[t_0,t_0+T)} \leq
\sqrt{\varepsilon}(1+e^{|k|\sqrt{\varepsilon}}+B)+\Delta_\xi.
\end{split}
\]
\end{lem}

\begin{proof}
Let $L$ be an arbitrary element of $[0,T]$. Noticing that $y(t)$ for $t\geq t_0 + L$, $L>0$, can be expressed as:
$y(t)=y(t-L)e^{ k L}+\int_{t-L}^t
e^{k(t-\tau)}(u(\tau)+d(\tau))d\tau$ and using the Mean-value
theorem we obtain: $y(t)-y(t-L)e^{ k L}= L
e^{k(t-\tau')}(u(\tau')+d(\tau')), \ \tau'\in[t-L,t]$. Hence  $ \varepsilon(1+e^{kL})\geq L e^{k(t-\tau')}
(|u(t)|-L B-\Delta_{\xi})$, and
\[
\begin{split}
&\begin{array}{l}\Delta_\xi+LB+\frac{\varepsilon(1+e^{kL})}{L \min\{1,e^{kL}\}}\end{array} \geq\\
&\begin{array}{l}\Delta_\xi+LB+\frac{\varepsilon(1+e^{kL})}{L \min\{1,e^{k(t-\tau')}\}}\end{array}
\geq |u(t)|  \ \forall t\geq t_0+L.
\end{split}
\]
Given that $L$ can be chosen arbitrarily in the interval $[0,T]$ we let
$L=\sqrt{\varepsilon}$, and thus $|u(t)|\leq
\sqrt{\varepsilon}(1+e^{k\sqrt{\varepsilon}})\max\{1,e^{-k\sqrt{\varepsilon}}\}+B\sqrt{\varepsilon}
+\Delta_\xi \leq
\sqrt{\varepsilon}(1+e^{|k|\sqrt{\varepsilon}}+B)+\Delta_{\xi}
\ \forall \ t\in [t_0 + \sqrt{\varepsilon},t_0+T]$.

Finally, given that $|\dot{u}(t)|\leq B$ for all $t\in[t_0,t_0+T]$, including in the interval $[t_0,t_0+\sqrt{\varepsilon}]$, we conclude that
\[
|u(t)|\leq
\sqrt{\varepsilon}(1+e^{|k|\sqrt{\varepsilon}}+2B)+\Delta_{\xi}
\ \forall \ t\in [t_0,t_0+T].
\]

\end{proof}

\subsection{Proof of Lemma \ref{lem:observer_inferrence}}

Let us rewrite (\ref{eq:system_io}) as
\[
\begin{split}
\dot{y}&= a_{1} y + \tilde{C}\tilde x + u_1(t)+d_1(t)\\
\dot{\tilde x}&= \tilde A \tilde  x + \tilde  a y + b u_1 +  G  u(t)+\tilde{ d}(t),
\end{split}
\]
where $\tilde{a}=\mathrm{col}(a_2,\dots,a_n)$,
$\tilde C=\mathrm{col}(1,0,\dots,0)$,
$\tilde d(t)=\mathrm{col}(d_2(t),\dots,d_n(t))$, and
\[
G=\left(\begin{array}{cc} - {b} & I_{n-1}
\end{array}\right),  \
\tilde  A=\left(\begin{array}{cc}0 & I_{n-2}\\
                                0 & 0
                                \end{array}\right).
\]
Let $\|y(t)\|_{\infty,[t_0,t_0+T]}\leq \varepsilon$ and denote
$e(t)=\tilde C^{T}\tilde x+u_1(t)$.

According to Lemma
\ref{lem:system_io_inferrence_1order},  there are
$\upsilon_1,\upsilon_2\in\mathcal{K}$ such that
$\|e(t)\|=\|\tilde C^{T}\tilde x+u_1(t)\|\leq
\upsilon_1(\varepsilon)+\upsilon_2(\Delta_\xi)$ for all $t\in[t_0,t_0+T]$.

Using the notation above we obtain:
$\dot{\tilde x}=(\tilde  A -   b
\tilde C^{T})\tilde x + \tilde  a y(t) +
\tilde{ G} u(t) +   b e(t)+\tilde{ d}(t)$.

Matrix $\tilde  A -   b \tilde C^{T}=\Lambda$ is Hurwitz,
and hence there are $D,k\in\Real_{>0}$ such that
$\|e^{\Lambda(t-t_0)}\|\leq D e^{-k (t-t_0)}$.
Therefore $ \|\tilde  C^{T} \tilde x(t)-
\tilde C^{T}\int_{t_0}^{t}e^{\Lambda(t-\tau)} G
 u(\tau)d\tau\| \leq De^{-k(t-t_0)}\|\tilde{ x}(t_0)\|
+\frac{D}{k}(\| a\|\varepsilon
+\| b\|(\upsilon_1(\varepsilon)+\upsilon_2(\Delta_\xi))+\Delta_\xi)
$.

Noticing that
$z_1=\tilde C^{T}\int_{t_0}^{t}e^{\Lambda(t-\tau)} G
 u(\tau)d\tau$, denoting $\kappa(\varepsilon)=2
\frac{D}{k}(\| a\|\varepsilon
+\| b\|\upsilon_1(\varepsilon))+\upsilon_1(\varepsilon)$,
$\kappa_2(\Delta_\xi)=2
\frac{D}{k}(\Delta_\xi+\| b\|\upsilon_2(\Delta_\xi))+\upsilon_2(\Delta_\xi)$,
and
\[
t'(\varepsilon,x_0)= t_0 + \frac{1}{k}\ln\left(\frac{D\|x_0\|}{\varepsilon}\right)
\]
we can conclude that there is a $t'(\varepsilon,x_0) \geq t_0$ such that
\[
\begin{split}
\|z_1(\tau)+ u_1(\tau)\|_{\infty,[t,t_0+T]}&\leq \kappa(\varepsilon)+\varepsilon+\kappa_2(\Delta_\xi)\\
&=\kappa_1(\varepsilon)+\kappa_2(\Delta_\xi).
\end{split}
\]
for all $t\in [t'(\varepsilon,x_0),t_0+T]$, providing that $T$ is sufficiently large to satisfy $t_0+T > t'(\varepsilon,x_0)$.

Noticing that $y(t)\equiv 0 \Rightarrow e(t)\equiv0$ ensures that (\ref{eq:identifiability}) holds too. $\square$

\end{document}